\documentclass[a4paper,12pt]{article}	

\usepackage[utf8x]{inputenc}     
\usepackage[british]{babel}      

\usepackage{amsmath}  
\usepackage{amsfonts} 
\usepackage{amssymb}  
\usepackage{amsthm}   

\usepackage{graphicx} 

\usepackage{comment} 

\usepackage{ae}      

\newcommand{\N}{\mathbb{N}}

\newcommand{\R}{\mathbb{R}}
\newcommand{\C}{\mathbb{C}}

\renewcommand{\le}{\leqslant}
\renewcommand{\ge}{\geqslant}

\DeclareMathOperator{\supp}{supp}

\newcommand{\ra}{\rightarrow}

\usepackage{enumerate}    

\theoremstyle{plain}
\newtheorem{theorem}{Theorem}[section]            
\newtheorem{proposition}[theorem]{Proposition}    

\newtheorem{corollary}[theorem]{Corollary}
\newtheorem{definition}[theorem]{Definition}

\newtheorem{remark}[theorem]{Remark}

\newtheorem{example}{Example}

\title{Quasi-isospectrality on quantum graphs}
\author{Ralf Rueckriemen }

\begin{document}




\nocite{BolteEndres09}
\nocite{BPB09}

\maketitle

\begin{abstract}
Consider two quantum graphs with the standard Laplace operator and non-Robin type boundary conditions at all vertices. We show that if their 
eigenvalue-spectra agree everywhere aside from a sufficiently sparse set, then the eigenvalue-spectra and the length-spectra 
of the two quantum graphs are identical, with the possible exception of the multiplicity of the eigenvalue zero. Similarly if their length-spectra agree everywhere aside from a sufficiently sparse set, then the quantum graphs have the same eigenvalue-spectrum and length-spectrum, again with the possible exception of the eigenvalue zero. 
\end{abstract}

\section{Introduction}

Let $G$ be a finite metric graph, that is a combinatorial graph where each edge is equipped with a positive real length.
Let the standard Laplace operator $\Delta=-\frac{\partial^2}{\partial x^2}$ act on the edges of $G$. Impose some 
boundary conditions (for example Kirchhoff-Neumann or Dirichlet) at all the vertices. 

A metric graph together with the operator and the boundary conditions is called a quantum graph. Quantum graphs are a popular 
model for various processes involving wave propagation in mathematics and physics. 
The Laplace operator on $G$ has an infinite discrete spectrum with finite multiplicities, the eigenvalue-spectrum of 
the quantum graph. The length-spectrum of $G$ consists of the lengths of all periodic orbits, each with a weight that depends 
on the boundary conditions at the vertices it passes through. 

One can now study the interplay between the quantum graph, its eigen\-value-spectrum and its length-spectrum. 

Asking how much information about the quantum graph is contained in the eigenvalue-spectrum is 
a well studied question. Under some suitable genericity conditions the spectrum determines the quantum graph uniquely \cite{GutkinSmilansky01}. On the other hand, there are numerous examples 
of Laplace-iso\-spectral non-isometric quantum graphs, see for example \cite{vonBelow99}, \cite{BSS06} or \cite{BPB09}. 

For the relationship between the eigenvalue-spectrum and the length-spectrum we can look at the manifold setting for inspiration.
Huber's theorem states that the eigenvalue-spectrum and the  length-spectrum determine each other on Riemannian manifolds of 
constant negative curvature (see \cite{Huber59} for the surface case and  \cite{DuistermaatGuillemin75} or \cite{Gangolli77} 
for the general case).

This is true for quantum graphs as well. 
Quantum graphs admit an exact trace formula, first proved in \cite{Roth83}, later also in \cite{KottosSmilansky99} and \cite{BolteEndres09} 
in much greater generality. It implies that the eigenvalue-spectrum and the length-spectrum determine each other.
Thus there are examples of pairs of quantum graphs that are Laplace-isospectral and length-isospectral but not isometric.

In this paper we will be concerned with \emph{quasi-isospectrality}. Assume two quantum graphs have the same eigenvalue-spectrum (including 
multiplicities) everywhere aside from some small exceptional set. We show that if this exceptional set has asymptotic density zero 
relative to the number of eigenvalues, then the two quantum graphs have to be Laplace-isospectral and length-isospectral, see theorem \ref{eigenvalue_spectrum} for the exact statement. In other words, quasi-isospectrality is impossible, either two quantum graphs are isospectral or their eigenvalue-spectra are quite different, there is no almost Laplace-isospectrality. The multiplidity of the eigenvalue zero is an exception, it can be different even if the rest of the eigenvalue-spectrum and the length-spectrum agree, see remark \ref{eigenvalue_zero}.

Next we consider the case of two quantum graphs having the same length-spectrum everywhere aside from some exceptional set. We show 
that if this exceptional set has asymptotic density zero relative to the length, then the two quantum graphs are length-isospectral 
and Laplace-isospectral, see theorem \ref{length_spectrum} for the exact statement. Again, there is the exception of the eigenvalue zero but  there is no almost length-isospectrality.

The same problem has been studied in the manifold setting. The case of hyperbolic 3-manifolds is considered in \cite{EGM98}. 
They show that 
a finite exceptional set implies length-isospectrality and Laplace-isospectrality both for the eigenvalue-spectrum and the length-spectrum. 
This was then generalized to hyperbolic manifolds of arbitrary dimension in \cite{BhagwatRajan10}, still with the restriction of a 
finite exceptional set. Finally, \cite{Kelmer11} shows the statements equivalent to ours for hyperbolic manifolds of arbitrary dimension.

This paper is structured as follows. First we introduce some notation for quantum graphs and define  
the non-Robin type boundary conditions 
we are going to use. Next we recall the trace formula for quantum graphs. We conclude the setup with the definition of the length-spectrum of a quantum graph. In section three and four we state and prove the two main theorems, first for the eigenvalue-spectrum 
and then for the length-spectrum.

\section{Setup}

\subsection{Quantum graphs and boundary conditions}

Let $G=(V,E, L)$ be a metric graph, $V$ is the set of vertices, $E$ the set of edges, each edge has a positive real length 
$L : E \ra \R_{>0}$ associated to it. We only consider finite metric graphs, all edge lengths are finite. The graphs are allowed to 
have loops and multiple edges.

The differential operator we consider is the standard Laplace operator acting as $\Delta=-\frac{\partial^2}{\partial x^2}$ on 
all edges. 

We will impose non-Robin type boundary conditions at all vertices. In particular, this will make the operator self-adjoint.
Non-Robin type boundary conditions do not mix conditions on the function with conditions on its derivative, 
examples include the most common boundary conditions such as Kirchhoff-Neumann or Dirichlet.

We will follow the treatment in \cite{KostrykinSchrader03}.
Assign an arbitrary orientation to all edges $1, \hdots, E$. Let $f : G \ra \C$ be a function on $G$, then we define its end values
\begin{align}
 F:=&(f_1(0), \hdots, f_E(0), f_1(L(1)), \hdots, f_E(L(E)))^T  \\
 F':=&(f_1'(0), \hdots, f_E'(0), -f_1'(L(1)), \hdots, -f_E'(L(E)))^T
\end{align}
We can write the boundary conditions as
\begin{align}
 AF + BF'=0
\end{align}
where $A,B  \in M(2E,\C)$. Non-Robin type boundary conditions are parame\-trized by pairs of matrices where $(A,B)$ has full 
rank and \mbox{$AB^*=0$}.
This parametrization is not unique. The matrices $A$ and $B$ have a block structure corresponding to the vertices. The 
Kirchhoff-Neumann 
boundary conditions at a vertex $v$ can be parametrized with the pair of matrices
\begin{align}
 A_v:= \begin{pmatrix} 1 & -1 && \\ & \ddots & \ddots & \\  & & 1 & -1 \\&&&0   \end{pmatrix} &&
B_v:= \begin{pmatrix} &&& \\ &&& \\  &&&\\ 1&1&1&1   \end{pmatrix}
\end{align}
where all not indicated matrix entries are zero.

\begin{definition}
A quantum graph is a metric graph equipped with a differential operator and some boundary conditions at the vertices. 
\end{definition}

\begin{proposition}
Given a quantum graph $G$ with non-Robin type boundary conditions at all vertices the Laplacian $\Delta$ is self-adjoint and has an infinite 
positive discrete spectrum $\{k_n^2\}_n$ with a single accumulation point at infinity. The multiplicity of each eigenvalue is finite. 
\end{proposition}
\begin{proof}
 This is well known, it follows from the fact that the Laplacian is elliptic and quantum graphs are compact.
\end{proof}

\begin{definition}
 We say two quantum graphs are {\bf Laplace-isospectral} if they have the same eigenvalue-spectrum, including multiplicities.
\end{definition}

\begin{definition}\label{S-matrix}
We define the {\bf $S$-matrix} of a quantum graph as
\begin{align} 
 S=S(A,B):=-(A+iB)^{-1}(A-iB)
\end{align}
\end{definition}

The conditions on the matrices $A$ and $B$ imply that this is well defined. This formula is a special case of the more general formula 
that includes a $k$ dependence and parametrizes Robin type boundary conditions as well, see \cite{KostrykinSchrader03}. 

 Let 
\begin{align}
 T(k):=\begin{pmatrix} 0 & t(k) \\ t(k) & 0 \\ \end{pmatrix} && t(k):= \begin{pmatrix} e^{ikL(1)} &&  \\ & \ddots &\\\ && e^{ikL(E)} \\ \end{pmatrix}
\end{align}
This matrix contains the metric information of $G$.

\subsection{The trace formula}

\nocite{GasquetWitomski99}

We require a suitable space of test functions for the trace formula. 
We are going to use the following.
\begin{definition}
\label{test_function}
A function $\varphi : \C \ra \C$ is called a test function if $\varphi$ is
\begin{enumerate}
\item even, that is $\varphi(z)=\varphi(-z)$
\item holomorphic on the strip $H:=\{ z=x+iy\in \C | -1\le y \le 1\}$
\item rapidly decreasing on $H$, for all $N\in \N$ there exists a constant $c_N$ such that
\begin{align}
 |\varphi(z)| < c_N(1+|z|)^{-N}
\end{align}
for all $z \in H$.
\end{enumerate}
\end{definition}

\noindent We denote the Fourier transform of $\varphi$ by 
\begin{align}
 \widehat{\varphi}(\xi)=\mathcal{F}(\varphi)(\xi)=\frac{1}{2\pi}\int_{\R}e^{i \xi x}\varphi(x)dx
\end{align}
and the inverse Fourier transform by 
\begin{align}
 \check{\varphi}(x)=\mathcal{F}^{-1}(\varphi)(x)=\int_{\R}e^{-i \xi x}\varphi(\xi)d\xi
\end{align}

\begin{theorem}[Paley-Wiener]
\label{Paley-Wiener} \cite{Hormander76}

If $\psi: \R \ra \R$ has compact support then $\widehat{\psi}$ can be extended to a function that is holomorphic on all of $\C$. 

If $\psi \in C^{\infty}(\R)$ and $\supp(\psi) \subset [-M,M]$ then for any $N \in \N$ there exists a $c_N$ such that
\begin{align}
| \widehat{\psi}(z)| \le c_N(1+|z|)^{-N}e^{M|Im(z)|}
\end{align}
In particular $\widehat{\psi}$ is rapidly decreasing on the strip $H$.
\end{theorem}

\begin{theorem}[Weyl law]
\cite{GnutzmannSmilansky06} or \cite{Grieser07}
\label{Weyl_law} 

Let $G$ be a quantum graph with non-Robin type boundary conditions at all vertices and spectrum $\{k_n^2\}_n$. Then one can estimate 
the number of eigenvalues in any interval $(K_0,K_1)$ by
\begin{align} 
 \left|\# \{ k_n | K_0 < k_n < K_1 \}- \frac{\mathcal{L}}{\pi}(K_1-K_0)\right| < 2E
\end{align}
\end{theorem}

\begin{definition}
Let $m_G(k)$ denote the multiplicity of the eigenvalue $k^2$ on $G$. Let \mbox{$m_G(k)=0$} if $k^2$ is not an eigenvalue of $G$.
\end{definition}

Consider the $\zeta$-function associated to the quantum graph $G$.
 \begin{align}
\zeta(k):=\det(S\cdot T(k))^{-1/2}\det(Id_{2n}-S\cdot T(k))  
 \end{align}

\begin{proposition}
\cite{KottosSmilansky99} and \cite{KostrykinSchrader06}
The $\zeta$-function is real on the real axis and holomorphic on a strip around the real axis.
For $k > 0$ the zeros of the $\zeta$-function correspond to the eigenvalues of the quantum graph $G$ including multiplicities.  
\end{proposition}

Let $N$ denote the multiplicity of the zero $k=0$ of $\zeta(k)$, in general this multiplicity is different from the multiplicity of the 
eigenvalue zero $m_G(0)$.

\begin{theorem}[the trace formula]
\cite{BolteEndres08} and \cite{BolteEndres09}

Let $G$ be a quantum graph with non-Robin type boundary conditions at all vertices.
 Then the spectrum $\{k_n^2\}_n$ of the Laplacian $\Delta$ determines the 
following exact trace formula.
\begin{align} \label{trace_formula_standard}
 \sum_{n=0}^{\infty}\varphi(k_n)=&\frac{\mathcal{L}}{2\pi}\widehat{\varphi}(0)+\left(m_G(0)-\frac{1}{2}N\right)\varphi(0)
 + \sum_{p \in PO}2Re(\mathcal{A}_p)\widehat{\varphi}(l_p)
\end{align}
Here $\varphi$ is a test function, see definition \ref{test_function}, and $\widehat{\varphi}$ is its Fourier transform. The second sum is over all 
periodic orbits, $l_p$ denotes the length of the periodic orbit $p$.

The coefficients $\mathcal{A}_p$ are given by
\begin{align} \label{periodic_orbit_coefficients}
 \mathcal{A}_p=\tilde{l}_p\prod_{b\in p} S_{bb'}
\end{align}
where $\tilde{l_p}$ is the length of the primitive periodic orbit $\tilde{p}$ that $p$ is a repetition of, $S_{bb'}$ is the coefficient in 
the $S$-matrix (see definition \ref{S-matrix}) that corresponds to the incoming and the outgoing oriented edge at that vertex.
\end{theorem}
The trace formula in the papers cited above is more general, their version holds for both Robin and non-Robin type boundary conditions.
There is a trace formula for non-Robin type boundary conditions in \cite{KPS07}, it corresponds to this one 
with the particular test function $\varphi(x)=e^{-tx^2}$. Another precursor to this trace formula is in \cite{KottosSmilansky99}, 
it is also in a distributional form but only allows for a specific set of boundary conditions.

\subsection{The length-spectrum of a  quantum graph}

We will now define the notion of length-spectrum of a quantum graph. 

The naive first idea would be to list all the lengths of periodic orbits and repeat 
lengths according to how many periodic orbits of the given length there are. Under this definition Huber's theorem does not hold for 
quantum graphs. 

The following example of two Laplace-isospectral quantum graphs is from \cite{BSS06}. 

\begin{figure}[ht]
\centering
\scalebox{0.4}{\includegraphics{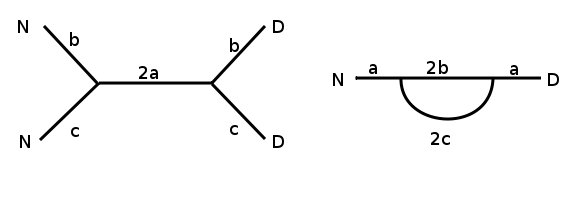}}
\caption{two isospectral quantum graphs}
\label{isospectral_example}
\end{figure}

The $N$ and $D$ stand for Neumann and Dirichlet boundary conditions at the vertex, all inner vertices have Kirchhoff-Neumann 
boundary conditions. The numbers $a,b,c$ are positive real numbers that correspond to the edge lengths.
The quantum graph on the right has periodic orbits of length $2a$ while the one on the left does not.

The proper  definition of length-spectrum assigns a weight to each periodic orbit that depends on the boundary conditions 
at the vertices.

\begin{definition}
\label{def_length_spectrum}
 The {\bf length-spectrum} of a quantum graph is the list of lengths of periodic orbits, each weighted with the factor 
\begin{align}
 A_G(l):= \sum_{p\in PO:{ } l_p=l} 2Re(\mathcal{A}_p)
\end{align}
where the coefficient $\mathcal{A}_p$ is from the trace formula, equation \ref{periodic_orbit_coefficients}. 
 If there are no periodic orbits of length $l$ we set $A_G(l)=0$.

If two quantum graphs have the same length-spectrum we say they are {\bf length-isospectral}.
\end{definition}

\begin{remark}
 Note that in the above definition the weight of a single periodic orbit and thus the weight at a particular length can be negative. 
The graph on the right in figure \ref{isospectral_example} has two periodic orbits of length $2a$ but their weights sum to 
zero, the two graphs are length-isospectral.
\end{remark}

\begin{corollary}[Huber's theorem]
 The length-spectrum and the eigenvalue-spectrum (without the eigenvalue zero) of a quantum graph determine each other.
 
 In particular, Laplace-isospectrality implies length-isospectrality and vice versa.
\end{corollary}
\begin{proof}
 This follows directly from the trace formula, equation \ref{trace_formula_standard}.
\end{proof}

\section{The eigenvalue-spectrum}

\begin{corollary}
\label{main_trace}
We can rewrite the trace formula as follows.
\begin{align}
\sum_{k \in \R_{> 0}}m_G(k)\varphi(k)=&\mathcal{L}\widehat{\varphi}(0)-\frac{1}{2}N\varphi(0)
 + \sum_{l \in \R_{>0}}A_G(l)\widehat{\varphi}(l)
 \end{align}
Note that $m_G(k)$ and $A_G(l)$ are nonzero on a discrete set, so if $\varphi$ is a test function both sides of the equation 
converge absolutely.  
\end{corollary}

\begin{remark}
 \label{eigenvalue_zero}
Notice that the multiplicity of the eigenvalue zero appears on both sides off the trace formula and cancels out. 

Consider the unit interval with either Dirichlet or Neumann boundary conditions at both ends. Then these two quantum graphs have the same length-spectrum and the same eigenvalue-spectrum except the eigenvalue zero, which has multiplicity one for Neumann boundary conditions and multiplicity zero for Dirichlet boundary conditions. 

This is an example of a more general phenomenon. Pick a quantum graph $G$ and create a new graph $G'$ by switching the roles of the matrices $A$ and $B$ in the boundary conditions. Then $S' = \overline{S}$ so these graphs have the same length-spectrum. Differentiating maps the eigenfunctions of one to the eigenfunctions of the other, thus $G$ and $G'$ have the same eigenvalue-spectrum away from zero. However, the multiplicity of the eigenvalue zero does not have to agree in general.
\end{remark}

\begin{remark}
We have 
\begin{align}
 \frac{1}{K}\sum_{k \in [0,K]}m_G(k) \sim \frac{\mathcal{L}}{\pi}
\end{align} 
as $K \ra \infty$ because $G$ has approximately $\frac{\mathcal{L}K}{\pi}$ eigenvalues in the interval $[0,K]$ by 
the Weyl law, theorem \ref{Weyl_law}.
\end{remark}

\begin{theorem}
\label{eigenvalue_spectrum}

 Let $G$ and $G'$ be two quantum graphs with non-Robin type boundary conditions at all vertices. 
If 
\begin{align}
\lim_{K \ra \infty} \frac{1}{K}\sum_{k \in [0,K]}\left|m_G(k)-m_{G'}(k)\right| = 0
\end{align}
then $G$ and $G'$ are Laplace-isospectral (with the possible exception of the multiplicity of the eigenvalue zero) and length-isospectral. Here $m_G(k)$ is the multiplicity of the eigenvalue $k^2$ of $G$.
\end{theorem}

\begin{proof}

If the limit above is zero, then $G$ and $G'$ have the same eigenvalue asymptotics, so by the Weyl law, theorem \ref{Weyl_law}, 
we have $\mathcal{L}=\mathcal{L}'$.

 Let $\psi \in C_c^{\infty}(\R)$ such that $\supp(\psi) \subset [-1,1]$, $\psi(0)=1$ and $\psi$ is even. 
Fix $l_0 \in \R$. Let
\begin{align}
 \psi_K(l):=\psi(K(l-l_0))+\psi(K(l+l_0))
\end{align}
for $K$ a large parameter (the idea to consider this kind of test function is adapted from \cite{Kelmer11}). Then 
$\check{\psi}_K(k)=2\pi \widehat{\psi_K}(k)$ because $\psi_K$ is even. 
By the Paley-Wiener 
theorem (theorem \ref{Paley-Wiener}), $\widehat{\psi_K}$ is holomorphic and rapidly decreasing. Thus $\check{\psi}_K$ is a valid test 
function (see definition \ref{test_function}) for the trace formula.
We have \mbox{$\check{\psi}_K(k)=\frac{4\pi}{K}\widehat{\psi}(\frac{k}{K})\cos(kl_0)$}. 
We plug $\check{\psi}_K$ in the difference of the trace formulas (corollary \ref{main_trace}) for $G$ and $G'$ and obtain
\begin{equation}\label{trace_K}\left. \begin{split}
&\frac{4\pi}{K}\sum_{k \in \R_{> 0}}\left(m_G(k)-m_{G'}(k)\right)\widehat{\psi}(\frac{k}{K})\cos(kl_0)\\
=&-\frac{1}{2}(N-N')\frac{4\pi}{K}\widehat{\psi}(0)
 + \sum_{l \in \R_{>0}}\left(A_G(l)-A_{G'}(l)\right)\psi_K(l) \end{split} \right.
\end{equation}

As $\widehat{\psi}$ is rapidly decreasing there exists a constant $C$ such that 
\mbox{$\left|\widehat{\psi}(x)\right| < \frac{C}{(1+|x|)^3}$}. We can bound 
\begin{align}
& \left| \frac{4\pi}{K}\sum_{k \in \R_{\ge 0}}\left(m_G(k)-m_{G'}(k)\right)\widehat{\psi}(\frac{k}{K})\cos(kl_0) \right| \\
<& \frac{4\pi}{K}\sum_{j=0}^{\infty} \sum_{k \in [jK,(j+1)K]}\left|m_G(k)-m_{G'}(k)\right|\cdot  \frac{C}{(1+(\frac{k}{K})^3} \\
<& 4C\pi\sum_{j=0}^{\infty} \sum_{k \in [jK,(j+1)K]}\frac{1}{K}\left|m_G(k)-m_{G'}(k)\right|\cdot \frac{1}{(1+j)^3} \\
\le& 4C\pi\sum_{j=0}^{\infty}\frac{1}{(j+1)^2} \sum_{k \in [0,(j+1)K]}\frac{\left|m_G(k)-m_{G'}(k)\right|}{(j+1)K} \\
\le& 4C\pi\sum_{j=0}^{\infty}\frac{1}{(j+1)^2} \sup_{K'\ge K}\sum_{k \in [0,K']}\frac{\left|m_G(k)-m_{G'}(k)\right|}{K'} \\
<& 4C\pi \frac{\pi^2}{6} \sup_{K'\ge K}\sum_{k \in [0,K']}\frac{\left|m_G(k)-m_{G'}(k)\right|}{K'} \\
\ra& 0
\end{align}
for $K \ra \infty$ by the initial assumption.

For $K$ sufficiently large we have
\begin{align}
& \sum_{l \in \R_{>0}}\left(A_G(l)-A_{G'}(l)\right)\psi_K(l) \\
=& \left(A(l_0,G)-A(l_0,G')\right)\psi_K(l_0)\\
=& \left(A(l_0,G)-A(l_0,G')\right)
\end{align}
because the $l$ with $A_G(l) \neq 0$ are discrete and $\psi_K$ is supported around $l_0$ and $-l_0$.
As all other terms in equation \eqref{trace_K} go to zero for $K \ra \infty$ this implies $A(l_0,G)=A(l_0,G')$ for all $l_0$.
Thus we have shown that $G$ and $G'$ are length-isospectral. If we look at the difference of the trace formulae \eqref{trace_K} 
and simplify we get
\begin{align}
&\sum_{k \in \R_{> 0}}\left(m_G(k)-m_{G'}(k)\right)\varphi(k)=-\frac{1}{2}(N-N')\varphi(0)
 \end{align}
By looking at a test function supported in a small enough neighborhood around $k$ we see that $m_G(k)=m_{G'}(k)$ for 
all $k\neq 0$ and $N=N'$ thus $G$ and $G'$ are Laplace-isospectral, with the possible exception of the eigenvalue zero.

\end{proof}

The following example shows that our bound on the exceptional set is best possible.

\begin{example}
Fix a value for $n$ and consider the following two quantum graphs, all edges have length one, all vertices have Kirchhoff-Neumann 
boundary conditions.

\begin{figure}[ht]
\centering
\scalebox{0.4}{\includegraphics{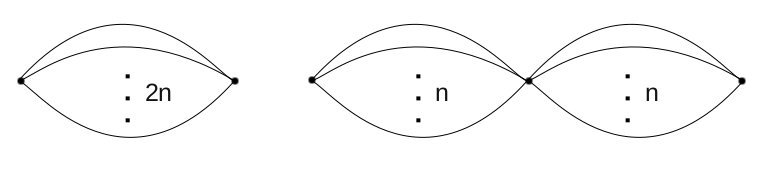}}
\caption{two quantum graphs whose spectra differ by a small proportion}
\end{figure}

Both quantum graphs are connected so $0$ is an eigenvalue with multiplicity one in both of them.

Let $k\in \N_{>0}$, then $k^2\pi^2$ is an eigenvalue of multiplicity $2n$ for the quantum graph on the left side. Assume all edges 
are parametrized from left 
to right as the interval $[0,1]$. Then there is one eigenfunction of the form $\cos(\pi k x)$ on all edges and a $(2n-1)$ dimensional 
eigenspace of functions of the form $a_e\sin(\pi k x)$ on each edge where the weights $a_e$ sum to zero. 

For the quantum graph on the right again assume that all edges are parametrized from left to right as the interval $[0,1]$ and 
$k\in \N_{>0}$. There is one eigenfunction of the form $\cos(\pi  k x/2)$ on the 
edges on the left and $\cos(\pi k x/2+1/2)$ on the edges on the right. There is a $(2n-2)$ dimensional eigenspace with eigenfunctions of the 
form $a_e\sin(\pi k x)$ on all edges where the sum of the $a_e$ over all edges on the right is zero and the sum over the $a_e$ over all 
edges on the left is zero.
Thus the eigenvalues are the numbers $k^2\pi^2$ for $k \in \N_{>0}$ with multiplicity $2n-1$ union the numbers $k^2\pi^2/4$ for $k$ an 
odd positive integer. 

This means these two quantum graphs satisfy
\begin{align}
\lim_{K \ra \infty} \frac{1}{K}\sum_{k \in [0,K]}|m_G(k)-m_{G'}(k)| = \frac{2}{n}
\end{align}
which can be made arbitrarily small by choosing $n$ big enough.

\end{example}

\section{The length-spectrum}

\begin{theorem}
\label{length_spectrum}
 Let $G$ and $G'$ be quantum graphs with non-Robin type boundary conditions at all vertices. If
\begin{align}
\lim_{L \ra \infty} \frac{1}{L}\sum_{l \in [0,L]}|A_G(l)-A_{G'}(l)| = 0
\end{align}
then $G$ and $G'$ are Laplace-isospectral (with the possible exception of the multiplicity of the eigenvalue zero) and length-isospectral.
\end{theorem}
\begin{proof}
  Let $\psi \in C_c^{\infty}(\R)$ such that \mbox{$\supp(\psi) \subset [-1,1]$}, \mbox{$\psi(0)=1$}, \mbox{$0 \le \psi(l)\le 1$} and $\psi$ is even. 
Fix $k_0 \in \R$. Let
\begin{align}
 \psi_L(l):=\frac{2}{L}\psi(\frac{l}{L})\cos(lk_0)
\end{align}
for $L$ a large parameter. Then $\check{\psi_L}(k)= 2\pi\widehat{\psi}(L(k-k_0))+2\pi\widehat{\psi}(L(k+k_0))$. 

By the Paley-Wiener 
theorem (theorem \ref{Paley-Wiener}), $\check{\psi}_L$ is a valid test 
function (see definition \ref{test_function}) for the trace formula.
We plug $\check{\psi}_L$ in the difference of the trace formulas (corollary \ref{main_trace}) for $G$ and $G'$ and obtain
\begin{equation}\label{trace_L} \left. \begin{split} 
&2\pi\sum_{k \in \R_{ > 0}}\left(m_G(k)-m_{G'}(k)\right)\widehat{\psi_L}(k)\\
=&(\mathcal{L}-\mathcal{L}')\psi_L(0) - \pi\left(N-N'\right)\widehat{\psi_L}(0)
 + \sum_{l \in \R_{>0}}\left(A_G(l)-A_{G'}(l)\right)\psi_L(l) \end{split} \right.
\end{equation}
For the term involving the length-spectrum we can estimate
\begin{align}
 &\left| \sum_{l \in \R_{>0}}\left(A_G(l)-A_{G'}(l)\right)\psi_L(l) \right|\\
 =& \left| \sum_{l \in \R_{>0}}\left(A_G(l)-A_{G'}(l)\right)\frac{2}{L}\psi(\frac{l}{L})\cos(lk_0) \right|\\
\le & 2\left| \sum_{l \in [0,L]}\frac{A_G(l)-A_{G'}(l)}{L} \right|\\
\ra & 0
\end{align}
as $L$ goes to infinity. We have $\psi_L(0)=\frac{2}{L} \ra 0$ as $L$ goes to infinity. Similarly 
$\widehat{\psi_L}(0)=\widehat{\psi}(k_0L)+\widehat{\psi}(-k_0L) \ra 0$ as $L$ goes to infinity because $\widehat{\psi}$ is rapidly 
decreasing.

This means we have 
\begin{align}
 \lim_{L\ra \infty}\sum_{k \in \R_{> 0}}\left(m_G(k)-m_{G'}(k)\right)\widehat{\psi_L}(k)= 0
\end{align}
because all other terms in equation \eqref{trace_L} go to zero in the limit \mbox{$L \ra \infty$}. We have 
$\widehat{\psi_L}(k_0)\ra 2\pi \widehat{\psi}(0) \neq 0$ and 
$\widehat{\psi_L}(k)\ra 0$ for all fixed $k \neq \pm k_0$ as $L \ra \infty$.
This implies $m_G(k_0)=m_{G'}(k_0)$ for all $k_0>0$. Thus $G$ and $G'$ have the same eigenvalue-spectrum with the possible exception 
of the eigenvalue zero, so by theorem \ref{eigenvalue_spectrum} they are Laplace-isospectral and length-isospectral.
\end{proof}

\begin{remark}
 Note that $\sum_{l \in [0,L]}A_G(l)$ grows exponentially in $L$ so presumably our bound is not best possible.
\end{remark}

\section{Acknowledgement}

This work was supported by a grant from EPSRC (grant EP/G021287/1).

\bibliographystyle{amsalpha}   
\addcontentsline{toc}{section}{Bibliography}   
\bibliography{literatur}

\end{document}